\documentclass[12pt]{amsart} 

\usepackage{amsmath}
\usepackage{amssymb}

\newcommand{\A}{\mathcal{A}}
\newcommand{\E}{\mathcal{E}}

\newcommand{\Z}{\mathbb{Z}}

\newcommand{\R}{\mathbb{R}}

\newtheorem{theorem}{Theorem}[section]
\newtheorem{definition}[theorem]{Definition}
\newtheorem{lemma}[theorem]{Lemma}
\newtheorem{conjecture}[theorem]{Conjecture}

\newtheorem{remark}[theorem]{Remark}

\newtheorem{example}[theorem]{Example}
\numberwithin{equation}{section}

\DeclareMathOperator{\PTE}{PTE}
\newcommand{\id}{\text{id}}

\newcommand{\word}[1]{\normalfont\mbox{\texttt{#1}}}

\makeatletter
\def\pt{\@ifstar\@@pt\@pt}
\def\@@pt#1{\text{partition } {\{#1\}}}
\def\@pt#1{{\{#1\}}}

\makeatother

\title[Prouhet-Tarry-Escott and Thue-Morse]{
The Prouhet-Tarry-Escott Problem \\
and \\
Generalized Thue-Morse Sequences 
}
\date{April 24, 2013}
\author[Bolker]{Ethan D. Bolker}
\address{Department of Computer Science and Department of Mathematics,
UMass Boston, Boston, MA 02125}
\email{eb@cs.umb.edu}
\urladdr{www.cs.umb.edu/$\sim$eb}

\author[Offner]{Carl Offner}
\address{Department of Computer Science, UMass Boston, 
Boston, MA 02125}
\email{offner@cs.umb.edu}

\author[Richman]{Robert Richman}
\email{richmanchemistry@gmail.com}

\author[Zara]{Catalin Zara}
\address{Department of Mathematics, UMass Boston, 
Boston, MA 02125}
\email{catalin.zara@umb.edu}
\urladdr{www.math.umb.edu/$\sim$czara}

\subjclass[2010]{05A18, 91B32}

\begin{document}
\maketitle

\begin{abstract}
We present new methods of generating  Prouhet-Tarry-Escott partitions of arbitrarily 
large regularity. One of these methods generalizes the construction of the Thue-Morse 
sequence  to finite alphabets with more than two letters. We show how 
one can use such partitions to (theoretically!) pour the same volume coffee from an urn 
into a finite number of cups so that each cup gets almost the same amount of caffeine.
\end{abstract}

\tableofcontents

\section{Introduction}

Mathematicians have studied the eponymous objects in our title for
more than a century and a half. We've stumbled on some generalizations
with interesting consequences and new open questions.

Our contribution to the ongoing story began with a query from Richman
asking about how he might generalize his solution \cite{Richman} to
the problem of pouring two cups of coffee of equal strength
from a carafe in which the concentration increases with depth to
three or more cups. 

To fill two cups with four pours use the \emph{word}
\word{ABBA}: pour the first and last quarters into cup \word{A} and
the second and third quarters into \word{B}. For eight pours the magic
word is \word{ABBA BAAB}. Continuing recursively by appending to each
sequence of length $n$ its complement (in the obvious sense) you find
the optimal partitions for pourings using $2^k$ subdivisions.
Collecting all the solutions into the infinite word 
\begin{equation*}
\word{AB BA BAAB BAABABBA } \ldots
\end{equation*}
produces the \emph{Thue-Morse sequence}.

Richman's argument showing (for example) that the word \word{ABBA
BAAB} solves the two cup problem using eighths depends essentially
on
\begin{equation*}
1 + 4 + 6 + 7  = 2 + 3 + 5 + 8
\end{equation*}
and
\begin{equation*}
1^2 + 4^2 + 6^2 + 7^2  = 2^2 + 3^2 + 5^2 + 8^2.
\end{equation*}

These equations say that the partition
$\pt{\pt{1,4,6,7}, \pt{2,3,5,8}}$ 
whose blocks are the positions of \word{A} and \word{B} in the
magic word solves an instance of the \emph{Prouhet-Tarry-Escott problem} -
finding partitions of a set of integers such that each block has the
same sum of powers for several powers. 

It's this connection we will generalize.

\section{Words and partitions}

We  set the stage with some formal definitions.

\begin{definition}
Let $S$ be a non empty set of integers and $r \geqslant -1$ an
integer.  A partition  $P = \pt{S_1, \ldots, S_b}$ of $S$ is
$r$-\emph{regular} if
\begin{equation*}
\sum_{x \in S_1} x^k = \sum_{x \in S_2} x^k = \dotsb 
= \sum_{x \in S_b} x^k\end{equation*}
for all $k=0,1, \ldots, r$. We write $\PTE(S,b,r)$ for
the set of all such partitions. A partition $P$ has \emph{maximal regularity} 
$r$ if it is $r$-regular but not $(r\!+\!1)$-regular.
\end{definition}

\begin{remark}
Every partition is $(-1)$-regular, so $\PTE(S,b,-1)$ is the
set of partitions of $S$ into $b$ blocks. Some of the blocks may 
be empty. 
\end{remark}

\begin{remark}
This definition and much of what follows makes sense over any ring,
not just $\mathbb{Z}$. 
\end{remark}

If $P=\pt{S_1, \ldots, S_b}$ is an $r$-regular partition of $S$ with
$r \geqslant 0$ then its blocks
have the same number of elements, and therefore $b$ divides $m = \#S$. 
Clearly 
\begin{equation*}
\emptyset = \PTE(S,b,m/b) \subseteq \dotsb \subseteq \PTE(S,b,1)
\subseteq \PTE(S,b,0).
\end{equation*}

\begin{lemma}\label{lem:affine}
(Affine invariance)
Let $n \neq 0$ and $a$ be integers. Define  
$f \colon \mathbb{Z} \to \mathbb{Z}$ by $f(x) = a+nx$. 
If $P = \pt{S_1, \ldots, S_b}$ partitions $S$ then 
\begin{equation*}
a+ nP := f(P) = \pt{f(S_1), \ldots, f(S_b)}
\end{equation*} 
partitions $a+nS$, and
\begin{equation*}
P \in \PTE(S,b,r) \Longleftrightarrow a+nP \in \PTE(a+nS, b, r) \; .
\end{equation*}
\end{lemma}
\begin{proof}
An easy induction on the powers less than or equal to $r$.
\end{proof}

In other words, regularity is invariant under affine transformations.

We are interested in the Prouhet-Tarry-Escott problem when $S$ is a
set of consecutive integers. Affine invariance implies that we
need consider just $S=[m]=\{1,\ldots, m\}$; we will write
$\PTE(m,b,r)$ for $\PTE([m], b,r)$. 
In that case, $b$-block partitions have natural string
representations over an alphabet $\A$ with $b$ letters $a_1,
\ldots, a_b$.

\begin{definition} The \emph{string representation} of a $b$-block
partition $P = (S_1, \ldots, S_b)$ of  $S = [m]$ is the 
$m$-letter word $a_1a_2 \ldots a_m$ where $a_i$ is the $t^{th}$ letter
of the alphabet when $i \in S_t$. 
\end{definition}

Conversely, given an $m$-letter word $w$ on a $b$ letter alphabet we can
construct the partition $P_w$ of $[m]$ using the equivalence relation that
defines two indices as equivalent when $w$ has the same letter in
those two places. 

For the letters in reasonably small alphabets we will use
$\word{A}, \word{B}, \word{C}, \ldots$ rather than subscripts $a_i$
or integers. We may also occasionally leave blanks between the letters
to emphasize features of interest. These have no semantic significance.

In what follows we will freely interchange partitions of $[m]$ and the
corresponding words. Some arguments are better in one language, some
in the other.

Permuting the letters of the alphabet corresponds to permuting the
order in which we write the blocks of the partition. Since that order
is essentially irrelevant, we will usually impose a particular
lexicographic order on the alphabet, and use letters in that order as
necessary starting at the beginning of a word.

In some studies of the Thue-Morse sequence and its generalizations
it's convenient to use the alphabet $\{0, 1, \ldots, m\!-1\!\}$. If 
you number the blocks of the partition with  those digits rather than
those in $[m]$ then the $m$-letter words that encode the partitions
can be viewed as integers written in base $m$. 

\section{A new class of solutions}

In this section we generalize the recursive construction of the
Thue-Morse sequence in order to generate a new family of solutions to
our Prouhet-Tarry-Escott problems.

\begin{definition}
A  \emph{Latin square} on a $b$-letter alphabet is a
$b \times b$ square matrix of letters such that each letter occurs
exactly once in each row and each column.  When we fix an order on 
the alphabet, a Latin square is
\emph{normalized} when its first column is in alphabetical order. 
\end{definition}

A Latin square can always be normalized by permuting its rows.
In the literature ``normalized'' sometimes means the columns are permuted
as well so that the first row is in alphabetical order. We do not
require that.  

\begin{example}\label{ex:latin}
There is only one normalized Latin square on a 2-letter alphabet,
\begin{equation*}
L_0 = 
\begin{bmatrix}  \word{A}& \word{B} \\
\word{B}& \word{A}
\end{bmatrix}\; .
\end{equation*} 

There are two 
normalized Latin squares on a 3-letter alphabet:
\begin{equation*}
L_1 = 
\begin{bmatrix} 
\word{A}&\word{B}&\word{C}\\
\word{B}&\word{C}&\word{A}\\
\word{C}&\word{A}&\word{B}
\end{bmatrix} 
\quad \text{ and } \quad
L_2 = 
\begin{bmatrix} 
\word{A}&\word{C}&\word{B}\\
\word{B}&\word{A}&\word{C}\\
\word{C}&\word{B}&\word{A}
\end{bmatrix}.
\end{equation*}
\end{example}

The columns of a normalized Latin square $L$ of size $b$ correspond to
a sequence of permutations $(\pi_1 = \id, \pi_2, \ldots,
\pi_b)$  such that for each row $x$ of $L$ the sequence
$(\pi_1(x) = x, \pi_2(x), \ldots, \pi_b(x))$ is a permutation of the
alphabet. We will often use that list of permutations to represent~$L$:
\begin{equation*}
L = (\id, \pi_2, \ldots, \pi_b).
\end{equation*}

Now we use normalized Latin squares to capture the essence of the
recursive construction of the Thue-Morse sequence.

\begin{definition}\label{def:Lw}
If $w= a_1 a_2\ldots a_m$ is an $m-$letter word and $\pi$ 
is a permutation of the alphabet, then $\pi(w)$ is the $m$-letter word 
\begin{equation*}
\pi(w) = \pi(a_1)\pi(a_2) \ldots \pi(a_m)\; .
\end{equation*}
When $L=(\id, \pi_2, \ldots,
\pi_b)$ is a normalized Latin  
square we write $L(w)$ for the concatenated $mb$-letter word 
\begin{equation*}
L(w) = w\pi_2(w) \dotsb \pi_b(w)\; .
\end{equation*}
If $P$ is the partition corresponding to word $w$ then we write
$L(P)$ for the partition corresponding the word $L(w)$.
\end{definition}


\begin{example}\label{ex:3.4}
With the notations of Example~\ref{ex:latin}, 
\begin{equation}\label{eq:sixteen}
L_0(\word{ABBA BAAB}) = \word{ABBABAAB BAABABBA}
\end{equation}
and
\begin{equation}\label{eq:six}
L_2(\word{AB}) = \word{ABCABC}.
\end{equation}
\end{example}

The motivation for Definition~\ref{def:Lw} is the fact that using a
Latin square this way increases the regularity of a partition. See \cite{allouche}, \cite{lehmer} 
for references to Prouhet's construction,  based on a Latin square action on a cycle of maximal length.

The partition corresponding to the word \word{AB} on the alphabet
$\{\word{A}, \word{B}, \word{C}\}$ 
is just $(-1)$-regular; 
Equation~\eqref{eq:six}  shows that it extends to 
word \word{ABCABC}, which corresponds to a $0$-regular partition.

The example in Equation~\eqref{eq:sixteen} is more interesting. The word
on the left encodes a $2$-regular partition. The one on the right
corresponds to
\begin{equation*}
\pt{\pt{1,4,6,7,10,11,13,16}, \pt{2,3,5,8,9,12,14,15}}\; ,
\end{equation*}
which is $3$-regular.
Here's the last step in the proof, assuming we've
already showed that it's $2$-regular.
Let $X$ be the sum of the cubes in the first block:
\begin{align*}
X & =  1^3 + 4^3 + 6^3 +7^3 + 10^3 + 11^3 + 13^3 + 16^3 \\
  & =  1^3 + 4^3 + 6^3 +7^3 + (2+8)^3 + (3+8)^3 +(5+8)^3 +(8+8)^3 \\
  & =  1^3 + 4^3 + 6^3 +7^3  \\
  & \ + 2^3 + 3(2^2\times 8) + 3(2 \times 8^2) +  8^3 \\
  & \ + 3^3 + 3(3^2\times 8) + 3(3 \times 8^2) + 8^3 \\
  & \ + 5^3 + 3(5^2\times 8) + 3(5 \times 8^2) + 8^3 \\
  & \ + 8^3 + 3(8^2\times 8) + 3(8 \times 8^2) + 8^3  \\
  & = \sum_{k=1}^8 k^3 + 24(2^2 + 3^2 + 5^2 + 8^2)
	+ 192(2+3+5+8) + 4(8^3).
\end{align*}
%
%
The same kind of computation shows
that the sum $Y$ of the cubes in the second block is
\begin{equation*}
Y = \sum_{k=1}^8 k^3 + 24(1^2 + 4^2 + 6^2 + 7^2)
	+ 192(1+4+6+7) + 4(8^3).
\end{equation*}
Since the partition corresponding to \word{ABBABAAB} is $2$-regular,
$X=Y$.

The formal proof of the general theorem calls for some machinery
that's a little more intricate than we like.

\begin{definition}
Let $L$ be a Latin square. Define its \emph{encoding matrix}
M = $\E(L)$ by 
\begin{equation*}
M_{ij} = x \Longleftrightarrow L_{jx} = i.
\end{equation*}
Thus $M_{ij}$ is the index of the column of $L$ in which the entry $i$
occurs  on row $j$:
\begin{equation*}
L_{j, M_{i,j}} = i \Longleftrightarrow M_{L_{i,j},i} = j.
\end{equation*}
\end{definition}

\begin{example}
If 
\begin{equation*}
L = 
\begin{bmatrix} 
\word{A}&\word{B}&\word{C}\\
\word{B}&\word{C}&\word{A}\\
\word{C}&\word{A}&\word{B}
\end{bmatrix} 
\end{equation*}
then 
\begin{equation*}
M = \E(L) = \begin{bmatrix} 1& 3 & 2 \\ 2 & 1 & 3 \\ 3 & 2 & 1\end{bmatrix}  .
\end{equation*}
\end{example}

\begin{theorem}\label{thm:Lw} Suppose $P$ partitions $[m]$ into $b$
blocks and $L$ is a normalized Latin square of size $b$.
\begin{enumerate}
\item If $P$ is $r$-regular, then $L(P)$ is $(r\!+\!1)$-regular.
\item If the encoding matrix $M = \E(L)$ is invertible and $L(P)$ is
$(r\!+\!1)$-regular then $P$ is $r$-regular. 
\item If $\E(L)$ is not invertible, then there exist partitions $P$
such that $L(P)$ is $1$-regular but $P$ is not $0$-regular. 
\end{enumerate}
\end{theorem}

\begin{proof}
Let $w=w_1w_2\ldots w_m$ be the word corresponding to $P$ 
on the alphabet $\A = \{a_1, a_2, \ldots, a_b\}$.  
For $j \geqslant 0$ and $x \in \A$ let
\begin{equation*}
S_{w,x}^{(j)} = \sum \{ t^j \mid w_t = x, 1 \leqslant t \leqslant
m\} .
\end{equation*}
Then $P$ is $r$-regular if and only if for every $j=0,\ldots, r$, 
the sum $S_{w,x}^{(j)}$ is the same for all $x \in \A$.

Then
\begin{align*}
S_{L(w), x}^{(j)} = & \sum \left\{ t^j \; | \; L(w)_t = x, 1
\leqslant t \leqslant bm\right\}  \\ 
= & \sum_{k=0}^{b-1} \left[ \sum \left\{(km+ t)^j \; | \;
L(w)_{km+t} = x, 1  \leqslant t \leqslant  m\right\} \right] \\
= & \sum_{k=0}^{b-1} \left[ \sum \left\{(km+ t)^j \; | \;
\pi_{k+1}(w_t) =x,   1 \leqslant t \leqslant  m\right\} \right] \\
= & \sum_{k=0}^{b-1} \left[ \sum \left\{ \sum_{i=0}^j \binom{j}{i}
(km)^{j-i}t^i \; | \; w_t =\pi_{k+1}^{-1}(x),   1 \leqslant t
\leqslant  m\right\} \right]  \\ 
= & \sum_{i=0}^j  \sum_{k=0}^{b-1} \binom{j}{i} (km)^{j-i}
S_{w,\pi_{k+1}^{-1}(x)}^{(i)} \; .
\end{align*}

Setting $x = a_s \in \A$,
\begin{equation*}
\pi_{k+1}^{-1} (a_s) = a_q \Longleftrightarrow s = \pi_{k+1}(q)
\Longleftrightarrow L_{q, k+1} = s \Longleftrightarrow k+1 =
M_{sq}.
\end{equation*}
Hence 
\begin{equation}\label{eq:long_eq}
\begin{aligned}
S_{L(w), a_s}^{(j)} & =   \sum_{q=1}^{b} S_{w, a_q}^{(j)} +  jm \sum_{q=1}^{b} (M_{sq}-1) S_{w, a_q}^{(j-1)} 
\\ & + \sum_{i=0}^{j-2}  \sum_{k=0}^{b-1} \binom{j}{i} (km)^{j-i} S_{w,\pi_{k+1}^{-1}(a_s)}^{(i)}   \\
& =  X(m,j) 
 + jm \sum_{q=1}^{b} M_{sq}S_{w, a_q}^{(j-1)}  \\
 & + \sum_{i=0}^{j-2}  \sum_{k=0}^{b-1} \binom{j}{i} (km)^{j-i} S_{w,\pi_{k+1}^{-1}(a_s)}^{(i)}\; , 
\end{aligned}
\end{equation}
where
\begin{equation*}
X(m,j) = \sum_{k=1}^m (k^j - jmk^{j-1})
\end{equation*}
is independent of $w$ and $s$.

If $w$ is $r$-regular then for every $i=0,\ldots, r$, the sum
$S_{w,y}^{(i)}$ is independent of $y$. Then  for all $j = 0, \ldots,
r+1$, the sum  $S_{L(w), a_s}^{(j)}$ does not depend on
$a_s$, which means that $L(w)$ has regularity $r+1$.

To prove (2), suppose that $L(w)$ is $(r\!+\!1)$-regular and $M$ is
invertible. 

There's nothing to prove if $r=0$, so we start with $r=1$. 

Let $Y_{w}^{(j)}$ be the column vector with entries $S_{w,x}^{(j)}$ for 
$x \in \A$ and $E$ the column vector with $b$ entries, all equal to 1. 
Then \eqref{eq:long_eq} implies
\begin{equation*}
Y_{L(w)}^{(1)} -X(m,1) E = mM Y_w^{(0)}\; .
\end{equation*}
If $L(w)$ is 1-regular, then the left hand side 
is a multiple of $E$. Since $E$ is an eigenvector of $M$, 
if $M$ is invertible, then the right hand side must also be a 
multiple of $E$, which shows that $w$ is 0-regular. Induction 
on $r$ using the same argument completes the proof of the second
statement.

For (3), suppose that $M$ is not invertible. Then its columns are
linearly dependent, so we can find integers $c_1, \ldots, c_b$
such that  
\begin{equation*}
c_1 \text{Col}_1(M) + \dotsb + c_b \text{Col}_b(M) = 0 \; .
\end{equation*}
Since the entries of $M$ are strictly positive, there will be both strictly 
positive and strictly negative values among $c_1, \ldots, c_b$. 
Pick a positive integer $h$ such that all the values $h+c_1, \ldots,
h+c_b$ are  non-negative and consider any word $w$ with $h+c_1$
letters $a_1$, $h+c_2$  letters $a_2$, and so on. Then $w$ is not
0-regular, but $L(w)$ is 1-regular. 
\end{proof}

For example, let $L$ be the Latin square 
\begin{equation}\label{eq:klein}
L= \begin{bmatrix} 
\word{A} &\word{B} &\word{C} &\word{D} \\
\word{B} &\word{A} &\word{D} &\word{C} \\
\word{C} &\word{D} &\word{A} &\word{B}  \\ 
\word{D} &\word{C} &\word{B} &\word{A}  
\end{bmatrix}
\simeq \begin{bmatrix} 1&2&3&4\\2&1&4&3\\3&4&1&2 \\ 4&3&2&1
\end{bmatrix} \; ,
\end{equation}
corresponding to the multiplication table for the Klein group 
$\mathbb{Z}_2 \times \mathbb{Z}_2$. In this example the matrix $\E(L)$ is the same as $L$ and is not
invertible. An example of a linear relation among the columns of $\E(L)$
is 
\begin{equation*}
\text{Col}_1 - \text{Col}_2 - \text{Col}_3 + \text{Col}_4 = 0\; ,
\end{equation*}
with coefficients $(1, -1, -1, 1)$ and a positive translate $(2, 0, 0,
2)$. Therefore any  
word $w$ with two A's and two D's generates a 1-regular $L(w)$,
even if  $w$ is not 0-regular.

\begin{enumerate}
\item $L(\word{ADAD})$ is 1-regular, but $\word{ADAD}$ is not 0-regular. 
\item $L(\word{BCCBADDA})$ is 2-regular but $\word{BCCBADDA}$ is only
0-regular. 
\end{enumerate}

When we first understood the first assertion of Theorem~\ref{thm:Lw}
we hoped it would generate all the solutions to our particular
Prouhet-Tarry-Escott problems. The third assertion dashed those hopes,
so we started to search for other constructions. You can read about
that in the next section. We close this one with some observations
providing examples where  $\E(L)$ is singular or invertible.

Notice that $\E$ has order three: $\E(\E(\E(L)))=L$ because
\begin{equation*}
\E(\E(\E(L)))_{i,j} = x \Leftrightarrow  \E(\E(L))_{j,x} = i
\Leftrightarrow \E(L)_{x,i} = j  
\Leftrightarrow L_{i,j}=x \;.
\end{equation*}
This periodicity allows us to reduce the problem of finding
Latin squares for which $\E(L)$ is singular or invertible to finding
Latin squares with those properties.

\begin{theorem}
For every positive integer $n$ there exist invertible Latin squares of
size $n$. 
\end{theorem}

\begin{proof}
Construct a Latin square $M_n$ of size $n$ by replacing $k$ by $k+1$
in the usual addition table of the group $\Z_n =\{ 0, 1,
\ldots,n\!-\!1\}$. 
After reversing the order of rows the corresponding matrix
becomes a circulant matrix with first row $(n, 1,2,\ldots, n\!-\!1)$, and
\begin{equation*}
| \det{M_n} | = \frac{(n+1)n^{n-1}}{2}\neq 0\; ,
\end{equation*}
hence $M_n$ is invertible.
\end{proof}
 
For example, when $n=6$ the Latin square $M_6$ is
\begin{equation}\label{eq:z6}
\begin{array}{c|cccccc}
& 0 & 1 & 2 & 3 & 4 & 5 \\
\hline
0 & 1 & 2 & 3 & 4 & 5 & 6 \\
1 & 2 & 3 & 4 & 5 & 6 & 1 \\
2 & 3 & 4 & 5 & 6 & 1 & 2 \\
3 & 4 & 5 & 6 & 1 & 2 & 3 \\
4 & 5 & 6 & 1 & 2 & 3 & 4 \\
5 & 6 & 1 & 2 & 3 & 4 & 5
\end{array} 
\end{equation}

\begin{theorem}\label{thm:noninvertible}
Let $n$ be a composite positive integer. Then there exist singular
Latin squares $L$ of size $n$. 
\end{theorem}

\begin{proof}
Let $a,b$ be integers such that $n=ab$ and $1 < a\leqslant
b$. Consider the addition table  $M$ of the group $\Z_a \times \Z_b$.
Enumerate the elements so that $(i,j)$ is the $(j\!+\!1\!+\!bi)^{th}$.
Then
\begin{equation*}
\text{Col}_1 - \text{Col}_2 - \text{Col}_{b+1} + \text{Col}_{b+2} = 0\; ,
\end{equation*}
hence $M$ is not invertible.
\end{proof}

The Latin square \eqref{eq:klein} corresponds to $a=b=2$. When $a=2$,
$b=3$ we obtain the singular normalized Latin square
\begin{equation}\label{eq:z2z3}
\begin{array}{c|cccccc}
& (0,0) & (0,1) & (0,2) & (1,0) & (1,1) & (1,2) \\
\hline
(0,0) & 1 & 2 & 3 & 4 & 5 & 6 \\
(0,1) & 2 & 3 & 1 & 5 & 6 & 4 \\
(0,2) & 3 & 1 & 2 & 6 & 4 & 5 \\
(1,0) & 4 & 5 & 6 & 1 & 2 & 3 \\
(1,1) & 5 & 6 & 4 & 2 & 3 & 1 \\
(1, 2) & 6 & 4 & 5 & 3 & 1 & 2
\end{array} 
\end{equation}
where $\text{Col}_1 +\text{Col}_{5} =  \text{Col}_2 + \text{Col}_{4}$.

\begin{remark}
Note that whether the addition table of a group is an invertible matrix or not 
depends on the order in which the elements are listed. Even though $\Z_6$ 
and $\Z_2 \times Z_3$ are isomorphic groups, the reordering of
elements that maps
\eqref{eq:z6} to \eqref{eq:z2z3} does not correspond to a group 
isomorphism.
\end{remark}

What happens when $n$ is prime? There are no singular Latin squares of
sizes 2 and 3 and a computer search indicates that all Latin squares
of size 5 are invertible, too. However, for $n=7$, the Latin square 
\begin{equation*}
\begin{bmatrix}
1&2&3&4&5&6&7\\
2&7&6&5&4&3&1\\
3&6&7&2&1&4&5\\
4&5&2&1&6&7&3\\
5&1&4&7&3&2&6\\
6&4&1&3&7&5&2\\
7&3&5&6&2&1&4
\end{bmatrix}
\end{equation*}
is singular.

\section{Changing the shapes of solutions}

In this section we study regularity-preserving operations on words.

\begin{theorem}\label{thm:swap}
Swap. Let $v$, $w$, $x$, $y$ and $z$ be words on a $b$-letter alphabet
such that $v$ and $w$ are ($r\!-\!1$)-regular and the concatenation
$xvywz$ is $r$-regular. Suppose either 
\begin{itemize}
\item $|v| = |w|$, or
\item $y$ is ($r\!-\!1$)-regular (possibly empty).
\end{itemize}
Then $xwyvz$ is also $r$-regular.
\end{theorem}

\begin{proof}
Left to the reader.
\end{proof}

\begin{theorem}\label{thm:alltwo}
There are $1$-regular words of length $n$ on a two letter alphabet if
and only if $n = 4k$. In that case
every element of $\PTE(4k, 2, 1)$ can be obtained from the word
\begin{equation*}
w = \word{A}^k \word{B}^{2k} \word{A}^k
\end{equation*}
by a sequence of swaps interchanging subwords \word{AB} and \word{BA}.
\end{theorem}

\begin{proof}
Let $v$ be a $1$-regular word of length $n$ on a two letter
alphabet. Then $n$ is even and since $v$ is $1$-regular, the block sums are equal, so 
\begin{equation*}
2(\text{block sum}) = \Sigma[n] = \frac{n(n+1)}{2} = \frac{n}{2} \times \text{odd}\; .
\end{equation*}
Then $n/2$ must also be even.

Conversely, it is clear that $w$ is a 1-regular word of length $n=4k$.

If the $1$-regular word $v = \word{\ldots BA\ldots AB\ldots} $
contains a subword $\word{BA}$ to 
the left of an $\word{AB}$ then Theorem~\ref{thm:swap} says
$v' = \word{\ldots AB\ldots BA\ldots} $ is also 1-regular and is
strictly less that $v$ in  lexicographic order. We can repeat this
procedure only a finite number of times, 
until we reach a 1-regular word $z$ with no subwords $\word{BA}$ to
the left of an $\word{AB}$.
Then $z$ is of the form 
$\word{A}^{p-1} \word{B}^{q} \word{A}\word{B}^{2k-q} \word{A}^{2k-p}$
for some $1 \leqslant p \leqslant 2k$ and $0 \leqslant q \leqslant 2k$. 
A straightforward computation shows that such a word is 1-regular if and only if 
$q=2k(p-k)$, hence $q=0, p=k$ or $q=2k, p=k+1$. Both imply $z=w$.
Reversing the sequence
of swaps changes $w$ into $v$.
\end{proof}

Swapping rearranges a word without changing either length or
regularity. Concatenation increases length, while preserving
regularity:

\begin{lemma}\label{lem:concatenation}
If words $v$ and $w$ correspond to $r$-regular partitions on a
$b$-letter alphabet then so does their concatenation $vw$.
\end{lemma}
\begin{proof}
Let $m$ be the length of $v$ and $n$ the length of $w$.
Lemma~\ref{lem:affine} shows that shifting word $w$ right by $m$
gives an $r$-regular partition of the integers between $n\!+\!1$ and 
$n\!+\!m$. 
The blocks of the partition corresponding to $vw$ are the unions of
corresponding blocks of $v$ and $w$. Since the component blocks from
each of $v$ and $w$ have the same sums of powers up to $r$, so do
their unions. 
\end{proof}

Splitting is the inverse of concatenation.

\begin{definition} ($k$-split) 
Let $w$ be an $r$-regular word on a $b$-letter alphabet -- that is, $w
\in \PTE(m,b,r)$. A $k$-split of $w$ is a list of $k$-regular
words $(w_1, w_2, \ldots, w_t)$ such that $w = w_1w_2 \cdots w_t$.
\end{definition}

The words $w_i$ need not have the same length. 
Lemma~\ref{lem:concatenation} implies that if $w$ can be $k$-split,
then it is $k$-regular.

\begin{example}
We can $k$-split the familiar $2$-regular \word{ABBABAAB} several ways
-- the blanks illustrate the subword boundaries:
\begin{equation*}
\word{ABBABAAB} = 
\word{ABBA BAAB} = 
\word{ABBA BA AB} = 
\word{AB BA BA AB} .
\end{equation*}
%
\end{example}

Theorem~\ref{thm:swap} implies that reordering the pieces of an
$(r\!-\!1)$-splitting  of a partition of regularity $r$ does not alter the
regularity below $r$. 

\begin{definition}
Let 
$(w_1, w_2, \ldots, w_t)$ be a list of words of the same length on the
same alphabet. The \emph{shuffle}
\begin{equation*}
w_1 \wedge w_2 \wedge\ldots \wedge w_t
\end{equation*}
of the list is the word $w$ built by concatenating the words built by
concating the $t$ first, second, \ldots letters of the $w_i$.
\end{definition}

\begin{example}
\begin{align*}
\word{AB} \wedge \word{BC} \wedge \word{CA} & = \word{ABC BCA} \\
\word{ABBA} \wedge \word{BAAB} & = \word{ABBABAAB} \\
\word{ABBA} \wedge \word{ABBA} & = \word{AABBBBAA} 
\end{align*}
\end{example}

\begin{theorem}\label{thm:shuffling}
(Shuffling)
The shuffle of $r$-regular words is $r$-regular.
\end{theorem}
\begin{proof}
Each component appears in the shuffle as an affine shift.
\end{proof}

Swapping, concatenation and shuffling are all methods of generating
new regular words from old. We have introduced these operations
in hopes that they will help find all the regular words from some
known ones, by analogy with Theorem~\ref{thm:alltwo}. 
There may be interesting questions to ask and answer about the algebra
of these operations -- the ways in which they associate, commute and
distribute.

\section{Existence}

\begin{theorem}\label{prop:5.1}
On a two-letter alphabet, there are $2$-regular words of 
length $n$ $\Longleftrightarrow$ $n = 4k$, with $k \geqslant 2$.
\end{theorem}

\begin{proof}
Suppose there are 2-regular words of length $n$. Theorem~\ref{thm:alltwo} 
implies that $n=4k$, since any 2-regular word is 1-regular. There are no 
2-regular words of length 4, hence $k \geqslant 2$. 

Conversely, suppose $n=4k$ with $k \geqslant 2$. Then $k$ can be written 
as a sum of $2s$ and $3s$, hence some concatenation of copies of the
2-regular words \word{ABBABAAB} and $\word{ABABBBAAABAB}$ 
generate a 2-regular $4k$-letter word.
\end{proof}

The $12$-letter  word $\word{ABABBBAAABAB}$ is a mystery.  
A computation similar to the one following Example~\ref{ex:3.4} shows
it is $2$-regular:
\begin{align*}
X & =  1^2 + 3^2 + 7^2 + 8^2 + 9^2 + 11^2 \\
  & = 1^2 + 3^2 + (1+6)^2 + (2+6)^2 + (3+6)^2 + (2+9)^2 \\
  & = 1^2 + 3^2  + 1^2 + 2(1 \times 6) + 6^2 + 2^2 + 2(2 \times 6) + 6^2 \\
  &  = 3^2 + 2(3 \times 6) + 6^2 + 2^2 + 2(2 \times 9) + 9^2 \\
  & = 2(1^2 + 2^2 + 3^2) + 6(2 + 4 + 6 + 6) + (3\times 6^2 + 9^2)
\end{align*}
while
\begin{equation*}
  Y = 2(1^2 + 2^2 + 3^2) + 6(1+2+3+3+9) + (3\times 3^2 + 2\times 9^2).
\end{equation*}

The word is $2$-regular because these expressions are equal -- term
by term. Why does that happen? 

\begin{theorem}
Let $r \geqslant 2$ and $n=k\cdot 2^r$, with $k \geqslant 2$. 
Then there exist $r$-regular words of length $n$ over a two-letter alphabet.
\end{theorem}

\begin{proof}
Induction on $r$. The base case $r=2$ is in
Theorem~\ref{prop:5.1}. The induction step follows from
Theorem~\ref{thm:Lw}. 
\end{proof}

A computer search shows that $\PTE(2, 2, 0)$, $\PTE(4,2,1)$, $\PTE(8,
2, 2)$, $\PTE(16, 
2, 3)$ each contain just one word,  
the initial segment of the Thue-Morse sequence of the corresponding length. 
Moreover, those are the minimal lengths of words with the respective
regularity. 

\begin{conjecture}\label{conj:maxreg}
Suppose $r \geqslant 2$. On a two-letter alphabet, there are
$r$-regular words of length $n$ $\Longleftrightarrow$ $n = k\cdot
2^r$, with $k \geqslant 2$. Moreover, $\PTE(2^{r+1}, 2, r)$ contains
just one word,  
the initial segment of the Thue-Morse sequence of length $2^{r+1}$.
\end{conjecture}

There are similar results for three-letter alphabets.

\begin{theorem} On a three-letter alphabet:
\begin{enumerate}
\item There are $1$-regular words of length $n$ $\Longleftrightarrow$
  $n = 3k$, with $k \geqslant 2$. 
\item There are $2$-regular words of length $n$ $\Longleftrightarrow$
  $n = 9k$, with $k \geqslant 2$. 
\end{enumerate}
\end{theorem}

\begin{proof}
Similar to the proof of the first part of Theorem~\ref{thm:alltwo}.
\end{proof}

A computer search shows that $\PTE(6,3,1)$  has one word
($\word{ABCCBA}$), $\PTE(18,3,2)$ has nine words, and $\PTE(36,3,3)$
has 152. Those are the minimum lengths of words of regularity 1, 2,
and 3 respectively. These numbers show that: 
\begin{enumerate}
\item There are 2-regular words of length 18 that do not come 
from a  Latin square construction starting with a 1-regular word of
length 6.
\item None of the 3-regular words of length 36 comes from a  Latin
  square  construction starting with a word of length 12, since the Latin
  squares of order 3 are invertible and there are no 2-regular words
  of length 12. 
\end{enumerate}

\begin{theorem}
Let $r \geqslant 3$ and $n = 2\cdot k\cdot  3^{r-1}$ with $k \geqslant 2$. 
Then there exists $r$-regular words of length $n$ over a three-letter alphabet.
\end{theorem}

\begin{proof}
Induction on $r$. For $r=3$ there are 3-regular words of $36=18\cdot 2$ and $54=18\cdot 3$ 
letters, hence, by concatenation, of any length of the form $18k$ with $k \geqslant 2$. 
The induction step follows from Theorem~\ref{thm:Lw}.
\end{proof}

\section{Resource allocation}

How does all this help answer the question of three or more cups of
coffee? We model the concentration of coffee in a cylindrical
cafeti\`{e}re as a function $f \colon [0,1] \to \R$. (In reality $f$
will increase with depth, but we won't need that.)
To fill $b$ cups of coffee with $m$ pours of equal size we want to
choose a partition $\{B_1, \ldots, B_b \}$
of the set of subintervals 
\begin{equation}\label{eq:partition01}
\left\{\left[0,\frac{1}{m}\right], \left[\frac{1}{m},
  \frac{2}{m}\right], \ldots , \left[\frac{m-1}{m},1\right] \right\}
\end{equation}
such that the integrals 
\begin{equation}\label{eq:cj}
c_j  = \int_{B_j} f(x)dx
 = \sum_{I \in B_j}\int_I f(x)dx
\end{equation}
are as nearly equal as possible.

We will identify the intervals in \eqref{eq:partition01}
by $m$ times their right endpoints, so the partitions of that set of
intervals are just the partitions of $\{1,2, \ldots, m \}$ we have been
studying.

\begin{theorem}
If $B \in \PTE(m,b,r)$ then the integrals in Equation~\eqref{eq:cj}
are independent of $j$ when $f$ is a polynomial of degree at most 
$r$. Therefore $B$ is a perfect pouring.
\end{theorem}

\begin{proof}
Consider first a monomial $f(x) = x^n$ for $ n \leqslant r$.
Using the change of variable $y=mx$  we have
\begin{equation*}
c_j = \frac{1}{m^{n+1}} \sum_{i \in B_j}\int_{i-1}^i y^n dy =  \frac{1}{(n+1)m^{n+1}} 
            \sum_{i \in B_j} \Bigl( i^{n+1} - (i-1)^{n+1} \Bigr) . 
\end{equation*}
But $i^{n+1} - (i-1)^{n+1}$ is a polynomial of degree $n$ in $i$ and 
since $B$ is $r$-regular and $n\leqslant r$, the last sum is
independent of $j$. Having proved the theorem for monomials its truth
follows easily for polynomials.\end{proof}

This argument may seem circular. It's not: the
theorem asserts the equality of integrals of sums of powers;
the last part of the last paragraph uses regularity to prove the
equality of sums of sums of powers.

When $f$ is not a polynomial we can use the first few terms of its
Taylor expansion to find pretty good pourings.

\begin{theorem}\label{thm:taylor}
Let $f \colon [0,1] \to \R$ be an $r\!+\!1$-times differentiable
function and suppose  $|f^{(r+1)}(x) | \leqslant M$ for all $0
\leqslant x \leqslant 1$. If $B \in \PTE(m,b,r)$ then
\begin{equation}\label{eq:abscicj}
| c_i - c_j| \leqslant \frac{M}{2^{r} b (r+1)!}\; .
\end{equation}
\end{theorem}

\begin{proof}
The Lagrange formula for the remainder of the
Taylor expansion of $f$ about $1/2$ says that
\begin{equation*}
f(x) = \text{a polynomial of degree } r
+ R(x)
\end{equation*}
where the error term satisfies
\begin{equation*}
| R(x)| = 
\Bigl| \frac{f^{(r+1)}(\xi_x)}{(r+1)!} \Bigl(x-\frac{1}{2}\Bigr)^{r+1} \Bigr|
 \leqslant 
\frac{M}{2^{r+1} (r+1)! }
\end{equation*}
for some $\xi_x$ between $0$ and $1$.
Then
\begin{equation}\label{eq:cicj}
c_i - c_j =  \int_{B_i} f(x)dx - \int_{B_j} f(x)dx
= \int_{B_i} R(x)dx - \int_{B_j} R(x)dx
\end{equation}
because the polynomial parts of the expansion of $f$ contribute the same
amount to the difference.
Each of the two terms in \eqref{eq:cicj} satisfies the inequality
\begin{equation*}
\Bigl| \int_{B_i} R(x)dx \Bigr|
\leqslant \frac{M}{2^{r+1} b (r+1)!}
\end{equation*}
since $B_i$ is the union of $m/b$ intervals each of length $1/m$.
Then their difference satisfies \eqref{eq:abscicj}.
\end{proof}

\begin{example} Suppose $f(x) = e^{-ax}$, with $a>0$. Then 
\begin{equation*}
\Bigl| f^{(r)}(x) \Bigr| = \Bigl| (-a)^r e^{-ax}\Bigr| \leqslant a^r\; .
\end{equation*}
Then the right side of \eqref{eq:abscicj} 
approaches $0$ as $r \to \infty$, so we have a strategy for pouring as
equitably as we wish by choosing a PTE solution with 
$r$ large enough.
\end{example}

The inequality in \eqref{eq:abscicj} provides a quantitative estimate
of the error of a particular pouring. Here is a more general
qualitative assertion:
\begin{theorem}\label{thm:analytic}
Suppose $f \colon [0,1] \to \R$ is analytic.
Then we can get a pouring as close to equitable as we want by choosing
a partition in $\PTE(m,b,r)$ for $r$ large enough.
\end{theorem}
\begin{proof}
The difference in remainders in Equation~\eqref{eq:cicj} can be made
arbitrarily small since $f$ is the uniform limit of the partial sums
of its power series.
\end{proof}

In \cite{levine-stange} the authors address resource allocations for
two players and remark that ``It would be interesting to quantify the 
intuition that the Thue-Morse order tends to produce a fair outcome.'' 
Theorem~\ref{thm:analytic} and Conjecture~\ref{conj:maxreg} show that
allocations tend to be more equitable as regularity increases, and
that the Thue-Morse sequence  produces the highest regularity for
words of fixed lengths that are powers of 2. 

In \cite{Richman} Richman showed that the Thue-Morse sequence provides
the most equitable pourings into two cups for a variety of density
functions $f$. Our analysis here does not extend his; all we show is
that regular partitions yield good pourings.

Should you ever actually use a regular partition for a pouring you can
take advantage of double letters in the word to save a few switches: 
\word{ABBABAAB} requires just $5$, not $7$. But don't get your hopes
up. The Thue-Morse sequence never contains \word{xxx}. That's probably
true for our generalizations, too. Nor are you  likely to find 
\word{xxyyzz}.%
\footnote{``bookkeeper'' is essentially the only English word we know
  that does.}

\end{document}